\documentclass[10pt]{elsarticle}

\usepackage{graphicx}
\DeclareGraphicsExtensions{.pdf,.png,.jpg}

\usepackage{amsmath,amssymb,amsfonts,amsthm,color}
\usepackage{booktabs,hhline,enumerate,subfigure,url}
\usepackage{algorithmic}
\usepackage[algoruled,vlined]{algorithm2e}

\journal{~}

\newcommand{\C}{\mathbb{C}}
\newcommand{\R}{\mathbb{R}}

\DeclareMathOperator{\tr}{tr}
\DeclareMathOperator{\grad}{grad}
\DeclareMathOperator{\e}{e}

\newcommand{\half}{\textstyle{\frac{1}{2}}}
\DeclareMathOperator{\ad}{ad}

\DeclareMathOperator{\DD}{D}
\DeclareMathOperator{\dist}{dist}

\DeclareMathOperator{\Gr}{Gr}
\DeclareMathOperator{\St}{St}

\DeclareMathOperator{\Herm}{Herm}
\DeclareMathOperator{\id}{id}

\DeclareMathOperator{\U}{U}
\DeclareMathOperator{\uu}{\mathfrak{u}}

\DeclareMathOperator{\einheit}{\mathcal{I}}

\let\oldmarginpar\marginpar
\renewcommand{\marginpar}[1]{\oldmarginpar{\textcolor{red}{\footnotesize{#1}}}}

\renewcommand{\top}{\mathsf{T}}		% matrix transpose
\newcommand{\hop}{\mathsf{H}}		% Hermitian transpose

\newtheorem{theorem}{Theorem}
\newtheorem{lemma}{Lemma}

\newtheorem{remark}{Remark}
\newtheorem{definition}{Definition}

\newcommand{\fd}[2]{\tfrac{\operatorname{d}}{\operatorname{d}{#2}}{#1}}

\begin{document}

\begin{frontmatter}

\title{
Averaging Complex Subspaces via a Karcher Mean Approach
}

\author[label1]{K. H\"uper \footnote{hueper@mathematik.uni-wuerzburg.de}}
\author[label2]{M. Kleinsteuber\footnote{Corresponding author, kleinsteuber@tum.de, web: www.gol.ei.tum.de }}
\author[label2]{H. Shen \footnote{hao.shen@tum.de \\ This work has partially been supported by the Cluster of Excellence \emph{CoTeSys - Cognition for Technical Systems}, funded by the German Research Foundation (DFG).}
}
\address[label1]{Department of Mathematics, Julius-Maximilians-Universit\"at  W\"urzburg, \\ 97074 W\"urzburg, Germany.}
\address[label2]{Department of Electrical Engineering and Information Technology, \\ Technische Universit\"at M\"unchen, 80333 Munich, Germany.}

\begin{abstract}
We propose a conjugate gradient type optimization
technique for the computation of the Karcher mean on
the set of complex linear subspaces of fixed dimension, modeled by the so-called
Grassmannian. 
The identification of the Grassmannian
 with Hermitian projection matrices allows an accessible
introduction of the geometric concepts required for an intrinsic conjugate
gradient method. In particular, proper
definitions of geodesics, parallel transport, and the Riemannian
gradient of the Karcher mean function are presented.
We provide an efficient step-size selection for the special case
of one dimensional complex subspaces and illustrate how the method can be employed for blind identification via numerical experiments.
\end{abstract}


\end{frontmatter}

%\linenumbers

%%%%%%%%%%%%%%%%%%%%%%%%%%%%%%%%%%%%%%%%%%%%%%%%%%%%%%%%%%%%%%%%%%%
%%                           SECTION 1                           %%
%%%%%%%%%%%%%%%%%%%%%%%%%%%%%%%%%%%%%%%%%%%%%%%%%%%%%%%%%%%%%%%%%%%
\section{Introduction}
In a wide range of signal processing applications and methods, subspaces of a fixed dimension play an important role. 
Signal and noise subspaces of covariance matrices are well studied objects
in classical applications, such as subspace tracking \cite{sriv:aap04} or
direction of arrival estimation \cite{adal:book10}.
More recently, a significant amount of work is focussed on applying subspace
based methods to image and video analysis \cite{tura:pami11}, as well as to matrix completion problems \cite{simo:bitnm10}.
One fundamental challenge amongst these works is the study of the statistical properties
of distributions of subspaces. 
Specifically, in the present work, we are interested in computing the mean of a
set of subspaces of equal dimension via averaging.

The averaging process, considered in this paper, employs the intrinsic
geometric  
structure of the underlying set and is also known as the computation of the 
\emph{Karcher mean} (in differential geometry, \cite{karcher:77}), 
Fr{\'e}chet mean or barycentre (statistics), geometric mean (linear algebra and matrix analysis), or
center of mass (physics).
General
concepts of a geometric mean have been
extensively studied from both theoretical and practical points of view. 
To mention just a few, they include probability theory and shape spaces 
\cite{Kendall:1990lr,Le:2001fk},
imaging \cite{Gramkow:2001la}, linear algebra and matrix analysis
\cite{Ando:2004lr}, interpolation
\cite{BussFillmore:2001aa}, and convex and differential geometry
\cite{Kendall:1991aa,CorcueraKendall:1999aa}.

An appropriate mathematical framework is given by the so-called
Grassmannian, which assigns a differentiable manifold structure to the set of subspaces of equal
dimension.
Usually, this is achieved by 
identification with a matrix quotient space.\footnote{The set of
  $m$-dimensional subspaces of $\C^n$ is identified with $\C^{n \times
    m}_*/GL(m)$, cf.\cite{absil:book}, $\C^{n\times m}_*$ is the set of full rank $(n \!\times\!
m)$-matrices, and $GL(m)$ are the complex invertible $(m \!\times\!
m)$-matrices. The equivalence relation is defined by $X\!\sim\!Y \!\Leftrightarrow\! X \!=\! g Y$ for some $g \in GL(m)$.}
In this work, we do not follow such an approach. By following \cite{helm:un07} instead,
we identify the set of subspaces of equal dimension with
a set of matrices. More precisely, we consider the set of Hermitian projectors of fixed rank, which inherits its differentiable structure from the surrounding vector space of Hermitian matrices. In contrast to \cite{helm:un07}, we consider the complex case here.
The identification of the complex Grassmannian with Hermitian projection matrices allows an accessible introduction of the geometric concepts such as geodesics, parallel transport, and the Riemannian gradient of the Karcher mean function.

In general, computing the Karcher mean on a smooth manifold involves a process 
of optimization, which by its own is of both theoretical and practical
interest. Various numerical 
methods have been developed on the Grass\-mannian, such as a 
direct method \cite{drei:tr06}, gradient descent algorithms \cite{absi:book08}, 
Newton's method \cite{helm:un07}, and conjugate gradient methods 
\cite{klei:icassp07,mitt:ivc12}.

In this work, we focus on the development of conjugate gradient methods. 
These methods have been proven to be efficient in many applications due to their trade-off between computational complexity and excellent convergence properties.
In particular, we propose an efficient step-size selection for the interesting case
where the Grassmannian is equal to the complex projective space. Moreover, we
outline how the developed method can be employed for blind identification.

The paper is organized as follows. %
Section~\ref{sec:02} recalls some basic concepts in differential geometry, which make
the present work intuitive and self-contained.
An abstract framework of conjugate gradient methods on smooth manifolds is given in 
Section~\ref{sec:03}. 
In Section~\ref{sec:04}, the geometry of the Grass\-mannian is presented, 
followed by a detailed analysis of the of the Karcher mean function in Section~\ref{sec:05}. 
A geometric CG algorithm is given in 
Section~\ref{sec:06} for the computation of the Karcher mean on the Grassmannian in general, together with a particularly efficient step-size selection for the special case of the complex projective space.
In Section~\ref{sec:07}, we outline how the proposed approach of averaging subspaces 
is evidenced to be useful in blind identification and a conclusion is drawn in Section~\ref{sec:08}. 

%%%%%%%%%%%%%%%%%%%%%%%%%%%%%%%%%%%%%%%%%%%%%%%%%%%%%%%%%%%%%%%%%%%
%%                           SECTION 2                           %%
%%%%%%%%%%%%%%%%%%%%%%%%%%%%%%%%%%%%%%%%%%%%%%%%%%%%%%%%%%%%%%%%%%%
\section{Differential geometric concepts}
\label{sec:02}

In this section, we shortly recall and explain the differential geometric concepts that are needed for this work.
We refer to \cite{Carmo:1992er} %,KobayashiNomizu:1963aa,KobayashiNomizu:1969aa} 
for a detailed insight into differential and Riemannian geometry and for the formal definitions of the mathematical objects, and to \cite{absil:book} for an introduction of the topic with a focus on matrix manifolds. 

Strictly speaking, a \emph{manifold} $M$ is a topological space that can locally be continuously mapped to
some linear space, where this map has a continuous inverse. These maps are called \emph{charts}, and since charts are invertible, we can consider the change of two charts around any point in $M$ as a local map from the linear space into itself. $M$ is a \emph{differentiable} or  \emph{smooth} manifold, if these maps are smooth for all points in $M$.
Many data sets considered in signal processing are subsets of such a
manifold. Important examples are matrix groups, the
set of subspaces of fixed dimension, the set of matrices with
orthonormal columns (so-called Stiefel manifold), the set of positive
definite matrices, etc.

To every point $x$ in the smooth manifold $M$ one can assign a
\emph{tangent space}, consisting of all velocities of smooth 
curves in $M$ that pass $x$. Formally, we define
\begin{equation}
T_xM:=\{ \dot{\alpha}_x(0)~|~\alpha(t) \subset M, \alpha_x(0)=x\}. 
\end{equation}
Intuitively, $T_xM$ contains all possible directions in which one can tangentially pass through $x$. The elements of $T_xM$ are called \emph{tangent vectors} at $x$.
 
A \emph{Riemannian manifold} $M$ is a smooth manifold with a scalar product $g_x( \cdot, \cdot )$ assigned to each tangent space $T_x M$ that varies smoothly with $x$, the so called \emph{Riemannian metric}. 
We drop the subscript $x$ if it is clear from the context which tangent space $g$ refers to.
The corresponding norm will be denoted by $\| \cdot \|_g$.
The Riemannian metric allows to measure the distance on the manifold. As a
natural extension of a straight line in the Euclidean space, a
\emph{geodesic} is defined to
be a smooth curve in $M$
that connects two sufficiently close points with shortest length. 
The length of a smooth curve $\alpha\colon (a,b) \to M$ on a Riemannian manifold is defined as
\begin{equation}
\label{eq:lengthfunctional}
L(\alpha)=\textstyle{\int_a^b
  \sqrt{g_{\alpha(t)}\left(\dot{\alpha}(t),\dot{\alpha}(t) \right)} \ \textrm{d}t}.
\end{equation}
In Euclidean space, two velocities at different locations are both vectors in this space. This allows
to form linear combinations and scalar products of these vectors.
In the manifold setting, however, this is not possible, since these velocities are elements in different (tangent) spaces. We hence need a way to identify 
tangent vectors at $x \in M$ with tangent vectors at $y \in M$ if
$x\ne y$. To
that end, we assume that there is a unique geodesic in $M$ that
connects $x$ and $y$, say $\gamma(t)$, with $\gamma(0)=x$ and
$\gamma(\tau)=y$, being possible if $x,y$ are not too far apart. 
The \emph{parallel transport} along $\gamma(t)$ admits one way of identifying 
$T_x M$ with $T_yM$.
A  rigorous definition is beyond the scope of this work, but loosely speaking,
the transportation is done in such a way that during the
transportation process, there is no intrinsic rotation of the transported vector. 
In particular, this
leaves the scalar product between the transported
vector and the velocity of the curve invariant.

Certainly, such an identification of tangent vectors depends on the
geode\-sic. Consider for example a sphere with two different geodesics
connecting the south with the north pole (i.e. two meridians) that leave the
south pole by an angle of  $\pi/2$. 
Parallel transporting the same vector along both meridians from the
south pole to the north will result in two antiparallel vectors at the north pole. 
Note that the identification of different tangent spaces via parallel transport along a geodesic is just \emph{one} particular instance of a more general concept termed \emph{vector transport} in \cite{absil:book}.

In order to minimize a real valued function on $M$,
we have to extend the notion of a gradient to the Riemannian manifold setting. To that end, recall that if $f \colon \R^n \to \R$ is smooth in $x$, there is a unique vector
$G(x)$ such that
\begin{equation}
\fd{f(x+t H)|_{t=0}}{t}=G(x)^\top H=:\langle G(x),H\rangle_{\mathrm{Euclid}} \quad \text{for all } H \in \R^n,
\end{equation}
where $(\cdot)^\top$ denotes transpose.
Typically, we write $\nabla f(x):=G(x)$ and call it the \emph{gradient} of $f$ at $x$. This coordinate free definition of a gradient can be straightforwardly adapted to the manifold case. 
Let
\begin{equation}\label{eq:fonM}
f \colon M \to \R
\end{equation}
be smooth in $x \in M$.
There is a unique tangent vector $G(x) \in T_xM$ such that 
\begin{equation}\label{eq:directional_der}
\operatorname{D}f(x) H := \fd{(f \circ \gamma)
  (t)|_{t=0}}{t}=g_x\left(G(x),\dot{\gamma}(0)\right) 
\end{equation}
for all geodesics $\gamma$ with $\gamma(0)=x$ and $\dot\gamma(0)=H$.
We denote the \emph{Riemannian gradient} as $\grad f(x):= G(x)$. 
Note, that the Riemannian gradient is a tangent vector in the
respective tangent space that depends on the chosen Riemannian
metric. It is unique due to the Riesz representation theorem.
The following special case is of particular interest. Let $M$ be a submanifold of some Euclidean space $E$ with Riemannian metric induced by the surrounding space, i.e. the Riemannian metric is obtained by restricting the scalar product from $E$ to the tangent spaces. In this setting, the \emph{normal subspace} is the orthogonal complement of the tangent space. Assume furthermore, that
the function $f$ that is to be minimized in \eqref{eq:fonM} is in fact the restriction of a function $\widehat{f}$ that is globally defined on the entire surrounding Euclidean space.
If $\Pi_x$ denotes the orthogonal projection from $E$ onto the tangent space $T_xM$, then the Riemannian gradient is just the projection of the gradient of $\widehat{f}$ in $E$. In formulas, this reads as
\begin{equation}
\label{eq:RGrad}
\grad f(x) = \Pi_x \nabla \widehat{f}(x).
\end{equation}

%%%%%%%%%%%%%%%%%%%%%%%%%%%%%%%%%%%%%%%%%%%%%%%%%%%%%%%%%%%%%%%%%%%
%%                           SECTION 3                           %%
%%%%%%%%%%%%%%%%%%%%%%%%%%%%%%%%%%%%%%%%%%%%%%%%%%%%%%%%%%%%%%%%%%%
\section{A conjugate gradient method on manifolds}
\label{sec:03}
In this section, we recall one possibility of how to transfer
the concept of conjugate gradient (CG) methods to the manifold
setting. We refer to \cite{absil:book} for a more general approach
that uses retractions on manifolds. Ultimately, the latter approach might lead to a whole
set of general methods to minimize \eqref{eq:fonM}.
%
%\begin{figure}[t]
% \begin{center}
%        \setlength{\unitlength}{1mm}
%        \begin{picture}(45, 33)
%   \qbezier(-25,17)(-1,20)(10,-2)
%   \qbezier(10,-2)(40,22)(65,9)
%   \qbezier(-25,17)(25,52)(65,9)
%%
%       \qbezier(-5,21)(25,30)(40,15)
%       %
%     \textcolor{blue}{         \put(1.8,22.7){\vector(4,1){14}}}
%   \textcolor{white}{\put(0,22.7){\vector(1,2){4}}}
%          \put(0,22.7){\circle*{1}}
%
%     \put(-14,18){\small{$\gamma_i(t)$}}
%     \put(10,27){\small{$H_i$}}
%          \put(2,20){$x_{i}$}
%     %
% %         \put(35.5,19.5){$x_{k+1}$}
%        \put(9.5,5){$M$}
%%            \linethickness{0.1mm}
%     %  \put(34.8,21.7){\vector(-3,-4){1}}
%      %    \qbezier(36,24.4)(35.3,22)(34.3,21.1)
%%              \linethickness{0.01mm}
%  \textcolor{blue}{        \put(22.5,29){\line(-3,1){19.3}}
%        \put(22.5,29){\line(-4,-3){15.3}}
%          \put(-12,24){\line(4,3){15.3}}
%          \put(-12,24){\line(3,-1){19.3}}
%          }
%{
%  \put(30.5,19.2){\circle*{1}}
%  \put(23,15){$x_{i+1}$} }
%{
%%
%\textcolor{blue}{        \put(50,15){\line(-4,5){12.5}}
%        \put(50,15){\line(-5,-3){15}}
%          \put(20,26){\line(4,1){17.5}}
%          \put(20,26){\line(3,-4){15}}
%          }
%\textcolor{blue}{   \put(26.8,19.2){\vector(3,-2){12}}}
% \textcolor{blue}{\put(24.5,19.2){\vector(2,0){10}}}
% \put(35,15){\small{$\tau(H_i)$}}
% \put(34,20){\small{$ G_{i+1}$}} }
%{
% \textcolor{red}{\put(22.3,19.2){\vector(2,3){6}}}
% \put(21,24){\small{$H_{i+1}$}}
% %
% }
%        \end{picture}         
%\end{center}
%\caption{Illustration of the geometric conjugate gradient method.}\label{fig3}
%\end{figure}
%
The CG method is initialized by some $x_0 \in M$ and the descent direction
$H_0:= -\grad f(x_0)$ is given by the Riemannian gradient.
Subsequently, 
sweeps are iterated that consist of two steps, a \emph{line search}
in a given direction (i.e. along a geodesic in that direction)
followed by an update of the \emph{search direction}. We illustrate the CG method on manifolds in Figure 
\ref{fig3}.
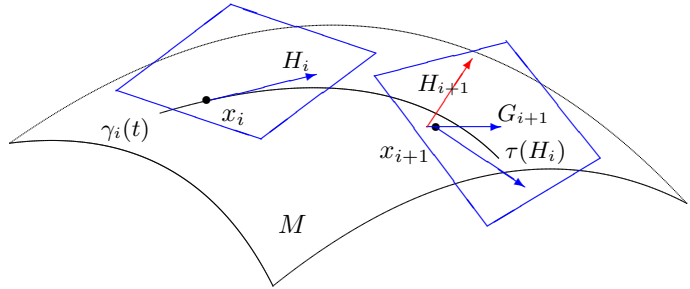
\begin{figure}[t]
 \begin{center}
        \setlength{\unitlength}{1mm}
        \begin{picture}(45, 33)
   \qbezier(-25,17)(-1,20)(10,-2)
   \qbezier(10,-2)(40,22)(65,9)
   \qbezier(-25,17)(25,52)(65,9)
       \qbezier(-5,21)(25,30)(40,15)
     \textcolor{blue}{         \put(1.8,22.7){\vector(4,1){14}}}
   \textcolor{white}{\put(0,22.7){\vector(1,2){4}}}
          \put(0,22.7){\circle*{1}}

     \put(-14,18){\small{$\gamma_i(t)$}}
     \put(10,27){\small{$H_i$}}
          \put(2,20){$x_{i}$}
     %
 %         \put(35.5,19.5){$x_{k+1}$}
        \put(9.5,5){$M$}
%            \linethickness{0.1mm}
     %  \put(34.8,21.7){\vector(-3,-4){1}}
      %    \qbezier(36,24.4)(35.3,22)(34.3,21.1)
%              \linethickness{0.01mm}
  \textcolor{blue}{        \put(22.5,29){\line(-3,1){19.3}}
        \put(22.5,29){\line(-4,-3){15.3}}
          \put(-12,24){\line(4,3){15.3}}
          \put(-12,24){\line(3,-1){19.3}}
          }
{
  \put(30.5,19.2){\circle*{1}}
  \put(23,15){$x_{i+1}$} }
{
\textcolor{blue}{        \put(50,15){\line(-4,5){12.5}}
        \put(50,15){\line(-5,-3){15}}
          \put(20,26){\line(4,1){17.5}}
          \put(20,26){\line(3,-4){15}}
          }
\textcolor{blue}{   \put(26.8,19.2){\vector(3,-2){12}}}
 \textcolor{blue}{\put(24.5,19.2){\vector(2,0){10}}}
 \put(35,15){\small{$\tau(H_i)$}}
 \put(34,20){\small{$ G_{i+1}$}} }
{
 \textcolor{red}{\put(22.3,19.2){\vector(2,3){6}}}
 \put(21,24){\small{$H_{i+1}$}}
 }
        \end{picture}         
\end{center}
\caption{Illustration of the geometric conjugate gradient method.}\label{fig3}
\end{figure}
Several different possibilities for these steps lead to different CG methods. Assume now that $x_i$, $H_i$, and $G_i:=\grad f(x_i)$
are given.

\subsection{Line search}
Given a geodesic $\gamma_i$ with $\gamma_i(0)=x_i$ and $\dot{\gamma}_i(0)=H_i$, the line search aims to find $a_i \in \R$ that minimizes $f \circ \gamma\colon t \to \R$.
We propose two approximations. The first is based on the assumption that $f \circ \gamma$
has its minimum near $0$, which under certain mild conditions follows from the fact that $x_i$ is already near the optimum. The step-size is chosen via a one dimensional Newton step, cf. \cite{klei:icassp07},
i.e.
\begin{equation}
\label{def:approx_stepsize_newton}
\textstyle{a_i^{\textrm{Newton}}:=-\tfrac{\ \tfrac{{\rm d}}{{\rm d}t} \left( f \circ \gamma \right) (t)|_{t=0}}{\big|{ \tfrac{{\rm d}^2}{{\rm d}t^2}}
 \left( f \circ \gamma\right) (t)|_{t=0}\big| }.}
\end{equation}
The absolute value in the denominator is chosen for the following reason. While being an unaltered one-dimensional Newton step in a neighborhood of a 
minimum 
the step size is the negative of a regular Newton step if \linebreak
$ \frac{{\rm d}^2}{{\rm d}t^2}
(f \circ \gamma)(t)\big|_{t=0} < 0$ and thus yields non-attractiveness for critical points that are no minima, cf. \cite{klei:controlo10}.

This approach, however, uses second order information of the cost function and is often computationally too expensive.
An alternative approach is the Riemmanian adaption of the backtracking
line search, described in Algorithm~\ref{algo:backtracking} below. The new iterate is then given by
\begin{equation}
x_{i+1}=\gamma(a_i),
\end{equation}
where $a_i$ is either obtained by backtracking or by Eq.
\eqref{def:approx_stepsize_newton}.
\begin{algorithm}
\caption{Backtracking line search on manifolds}\label{algo:backtracking}
	\SetAlgoNoLine
	\SetKwHangingKw{SI}{Step 1:}
	\SetKwHangingKw{SII}{Step 2:}
	\SetKwHangingKw{SIII}{Step 3:}
	
	\SI{Choose $\overline{a} > 0,\ c,\rho \in (0,1)$ and 
		set $a \gets \overline{a}$}
	
	\SII{{\bf repeat} until $(f \circ \gamma)(a) \leq 
			f(x_i) + c\ a\ g_x(G_i, H_i)$}
		{\hspace{20mm}{$a \gets \rho\ a$}\;}
		{\hspace{16mm}{\bf end repeat};}
	
	\SIII{Choose step-size $a_i^{\textrm{backtrack}}:=a$}
\end{algorithm}
\subsection{Search direction update}

In order to compute the new search direction $H_{i+1} \in T_{x_{i+1}}M$, we need to 
transport $H_i$ and $G_i$, which are tangent to $x_i$, to the tangent space $T_{x_{i+1}}M$.
This is done via parallel transport along the geodesic $\gamma$, which we denote by
\begin{equation}
\tau \colon T_{x_{i}}M \to T_{x_{i+1}}M.
\end{equation}
The updated search direction is now chosen 
according to a Riemannian adaption of
the Hestenes-Stiefel formula, or any other CG formula known from the Euclidean case, cf. \cite{noce:no99}.
Specifically, we have
\begin{align}
H_{i+1}=-G_{i+1}+r_i \tau H_i,
\end{align} 
where the most common formulas for $r_i$ read in the manifold
setting as 
\begin{equation}
\begin{aligned}\label{eq:different gammas}
r_i^{HS}&=\tfrac{ g( G_{i+1}, G_{i+1}- \tau G_{i} )}
{g (\tau H_i, G_{i+1}-\tau G_{i})} &&\text{(Hestenes-Stiefel)}  \\[.2cm]
r_i^{PR}&= \tfrac{ g( G_{i+1}, G_{i+1}- \tau G_{i})}
{\|G_i\|_g^2} &&\text{(Polak-Ribi\`{e}re)} \\[.2cm]
r_i^{FR}&= \tfrac{ \| G_{i+1}\|_g^2 }
{\|G_i\|_g^2}  &&\text{(Fletcher-Reeves)} \\[.2cm]
r_i^{DY}&= \tfrac{  \| G_{i+1}\|_g^2}
{g (\tau H_i,G_{i+1}-\tau G_{i})} &&\text{(Dai-Yuan)} \\[.2cm]
r_i^{*}&= - \tfrac{ g( G_{i+1}, G_{i+1}- \tau G_{i} )}
{g ( H_i , G_{i})}.
\end{aligned} 
\end{equation}
Albeit the nice performance in applications, convergence 
analysis of CG methods on smooth manifolds is still an open problem. To the best 
of the authors' knowledge, the only partial convergence result is
provided in 
\cite{Smith:1994gb}.

%%%%%%%%%%%%%%%%%%%%%%%%%%%%%%%%%%%%%%%%%%%%%%%%%%%%%%%%%%%%%%%%%%%
%%                           SECTION 4                           %%
%%%%%%%%%%%%%%%%%%%%%%%%%%%%%%%%%%%%%%%%%%%%%%%%%%%%%%%%%%%%%%%%%%%
\section{Geometry of the Grassmannian}
\label{sec:04}
The complex Grassmannian $\Gr_{m,n}$ is
defined as the set of complex $m$-dimensional $\C$-linear subspaces of
$\C^n$. It 
provides a natural generalization of the familiar complex projective
spaces. We denote the unitary group by
\begin{equation}
  \label{eq:5}
\U_{n}:=\{X\in\C^{n\times n}|X^{\hop}X=I_n\},
\end{equation}
where $(\cdot)^\hop$ denotes complex
conjugate transpose,  and $I_n$ is the $(n \times n)$-identity matrix. For computational purposes it makes sense to identify  the Grass\-mannian $\Gr_{m,n}$ with a set
of self-adjoint Hermitian projection operators as
\begin{equation}
  \label{eq:7}
  \Gr_{m,n}:= \{P \in \C^{n\times n}\ | \ P^{\hop}=P, P^2=P,\tr P=m\},
\end{equation}
i.e. the smooth manifold of rank $m$ Hermitian projection operators of
$\C^n$. Here, $\tr(\cdot)$ is the trace of a matrix.
In the sequel we describe the Riemannian
geometry directly for the submanifold $\Gr_{m,n}$
of $\C^{n\times n}$. As we will see, this approach has advantages
that simplify both the analysis and the design of CG-based
algorithms. We begin by recalling facts about the complex 
Grassmannian \cite{Husemoller:1993hl,Onishchik:1994pi}. Let
\begin{equation}
  \label{eq:8}
\uu_{n}\!:=\!\{\Omega\in \C^{n\times n} |\Omega^{\hop}\!=\!-\Omega\}\quad\text{and}\quad
\Herm_n \!:=\!\imath \uu_{n} 
\end{equation} 
denote the \emph{real} $n^2$-dimensional vector spaces of
skew-Hermitian and Hermitian matrices, respectively. Here we
  follow the terminology in group theory, with $\uu_n$ being the Lie
  algebra for the unitary Lie group $\U_n$. In particular, $\e^{\uu_n}=\U_n$,
 where $\e^{( \cdot )}$ is the matrix exponential function.
\begin{theorem}
The Grassmannian $\Gr_{m,n}$ is a real, smooth, and 
compact submanifold of $\Herm_{n}$ of real dimension $2m(n-m)+1$.
Moreover, the tangent space at an
element $P\in \Gr_{m,n}$ is given as
\begin{equation}
  \label{eq:11}
  T_P\Gr_{m,n}=\{P\Omega-\Omega P\ | \ \Omega\in \uu_{n}\}.
\end{equation}
\end{theorem}
It is useful for further analysis to define the linear operator 
\begin{equation}
  \label{eq:12}
 \ad_P:\C^{n\times n} \to \C^{n\times n},\quad
\ad_P(X):=[P,X]:=PX-XP.
\end{equation}
\begin{lemma} 
\label{lem:adP3}
(\cite{helm:un07} for the real case)
For any $P\in \Gr_{m,n}$ the minimal polynomial of
$\ad_P$ is equal to
$s^3 - s$. Thus $\ad_{P}^3=\ad_{P}$, i.e., 
\begin{equation}
  \label{eq:13}
  \ad_{P}^2H =[P,[P,H]]=H\quad\forall H \in
T_{P}\Gr_{m,n}.
\end{equation}
\end{lemma}
In the sequel, we will always endow
$\Herm_{n}$ with the Frobenius inner product, defined by
\begin{equation}
  \label{eq:98}
\langle X,Y\rangle := \tr(XY).
\end{equation}
 The Euclidean Riemannian
metric $g_P$ on $\Gr_{m,n}$ induced by the embedding space
$\Herm_{n}$ is defined by the restriction of \eqref{eq:98} to the
tangent spaces, i.e.
\begin{equation}
  \label{eq:97}
  g_P(H_1,H_2) = \tr (H_1 H_2)\quad\forall P \in \Gr_{m,n} \text{ and }\forall H_1, H_2\in T_{P}\Gr_{m,n}. 
\end{equation}
\begin{lemma} 
\label{lem:normalGR}
(\cite{helm:un07} for the real case)
Let $P\in \Gr_{m,n}$ be arbitrary.
The normal subspace at $P$ in $\Herm_{n}$ is given by
$
  N_{P}\Gr_{m,n}=\{X-\ad_{P}^2X \ | \ X\in \Herm_{n}\}.
$
The linear map
\begin{equation}
  \label{eq:49}
  \Pi_P: \Herm_{n}\to \Herm_{n},\; X\mapsto \ad_{P}^{2}X=[P,[P,X]]
\end{equation}
is the self-adjoint Hermitian projection operator onto
$T_{P}\Gr_{m,n}$ with kernel $N_{P}\Gr_{m,n}$.
\end{lemma}
In general, a geodesic is a minimizer of
the variational problem (\ref{eq:lengthfunctional}), i.e. it is the
solution of the corresponding Euler-Lagrange equation, the latter
being a second order ordinary differential equation. The following result characterizes
the geodesics on $\Gr_{m,n}$. 
\begin{theorem} 
\label{theo:GeoGrass}
(\cite{helm:un07} for the real case)
The geodesics of $\Gr_{m,n}$ are
exactly the solutions of the second order differential equation
$
\ddot{P} + [\dot{P},[\dot{P},P]] = 0
$.
The unique geodesic $P(t)$ with initial conditions $P(0)=P_0 \in
\Gr_{m,n}$, $\dot{P}(0)=H \in
T_{P_0}\Gr_{m,n}$ is given by
\begin{equation}\label{eq:geodesic2}
P_H(t) = \e^{t[H,P_0]}P_0\e^{-t[H,P_0]}.
\end{equation}
\end{theorem}
\begin{definition}
We define the \emph{Riemannian exponential map} as
\begin{equation}
\exp_{P_0} \colon T_{P_0} \Gr_{m,n} \to \Gr_{m,n}, \quad H \mapsto P_H(1).
\end{equation}
\end{definition}

As outlined above, we need the concept of
parallel transport along geodesics to give vector addition a
well defined meaning. 
\begin{lemma}
[\cite{HuperSilva-Leite:2007aa} for the real case] For $P\in \Gr_{m,n}$ and $G_0\in T_P\Gr_{m,n}$
the parallel transport of $G_0$ along the geodesic $P_H(t)$  is
given by
\begin{equation}
  G_H(t)=\e^{t[H,P]}G_0\e^{-t[H,P]}.
\end{equation}
\end{lemma}

Note, that the complex case considered here follows by a straightforward adaption of the proof of the real case in \cite{HuperSilva-Leite:2007aa}.

%%%%%%%%%%%%%%%%%%%%%%%%%%%%%%%%%%%%%%%%%%%%%%%%%%%%%%%%%%%%%%%%%%%
%%                           SECTION 5                           %%
%%%%%%%%%%%%%%%%%%%%%%%%%%%%%%%%%%%%%%%%%%%%%%%%%%%%%%%%%%%%%%%%%%%
\section{Karcher mean}
\label{sec:05}
\subsection{The distance between two complex subspaces}

In the first step we investigate the Riemannian distance of two complex subspaces.
For convenience, we denote the standard projector by
\begin{equation}
\einheit:= \left[
  \begin{smallmatrix}
    I_m&0\\0&0
  \end{smallmatrix}
\right].
\end{equation}
Let $P \in \Gr_{m,n}$, and assume for the moment that $P$ is sufficiently close to $\einheit$, i.e.
that there is a \emph{unique} geodesic emanating from $\einheit$ to $P$. 
Together with Eq. \eqref{eq:geodesic2} this implies the existence of a \emph{unique}
 $Z=Z(P)\in
B_0\subset\C^{(n-m)\times n}$ where $B_0$ is a sufficiently small open
ball around the zero matrix $0$,
such that
\begin{equation}
  \label{eq:2a}
  P=\e^{\left[
  \begin{smallmatrix}
    0&-Z^\hop\\Z&0
  \end{smallmatrix}
\right]}\einheit \e^{-\left[
  \begin{smallmatrix}
    0&-Z^\hop\\Z&0
  \end{smallmatrix}
\right]}.
\end{equation}
Hence, $Z$ can be considered as a function
of $P$, implicitly defined by \eqref{eq:2a}. 
\begin{figure}[t]
 \begin{center}
        \setlength{\unitlength}{1mm}
        \begin{picture}(45, 33)
   \qbezier(-25,17)(-1,20)(10,-2)
   \qbezier(10,-2)(40,22)(65,9)
   \qbezier(-25,17)(25,52)(65,9)
       \qbezier(-5,21)(25,30)(35,15)
     \textcolor{blue}{         \put(1.8,22.7){\vector(4,1){14}}}
   \textcolor{white}{\put(0,22.7){\vector(1,2){4}}}
          \put(0,22.7){\circle*{1}}

     \put(-14,18){\small{$\gamma_i(t)$}}
     \put(10,27){\small{$H_i$}}
          \put(2,20){$x_{i}$}
     %
 %         \put(35.5,19.5){$x_{k+1}$}
        \put(9.5,5){$M$}
%            \linethickness{0.1mm}
     %  \put(34.8,21.7){\vector(-3,-4){1}}
      %    \qbezier(36,24.4)(35.3,22)(34.3,21.1)
%              \linethickness{0.01mm}
  \textcolor{blue}{        \put(22.5,29){\line(-3,1){19.3}}
        \put(22.5,29){\line(-4,-3){15.3}}
          \put(-12,24){\line(4,3){15.3}}
          \put(-12,24){\line(3,-1){19.3}}
          }
{
  \put(30.5,19.2){\circle*{1}}
  \put(23,15){$x_{i+1}$} }
{
\textcolor{blue}{        \put(50,15){\line(-4,5){12.5}}
        \put(50,15){\line(-5,-3){15}}
          \put(20,26){\line(4,1){17.5}}
          \put(20,26){\line(3,-4){15}}
          }
\textcolor{blue}{   \put(26.8,19.2){\vector(3,-2){12}}}
 \textcolor{blue}{\put(24.5,19.2){\vector(2,0){10}}}
 \put(35,15){\small{$\tau(H_i)$}}
 \put(34,20){\small{$ G_{i+1}$}} }
{
 \textcolor{red}{\put(22.3,19.2){\vector(2,3){6}}}
 \put(21,24){\small{$H_{i+1}$}}
 }
        \end{picture}         
\end{center}
\caption{Illustration of the geometric conjugate gradient method.}\label{fig3}
\end{figure}
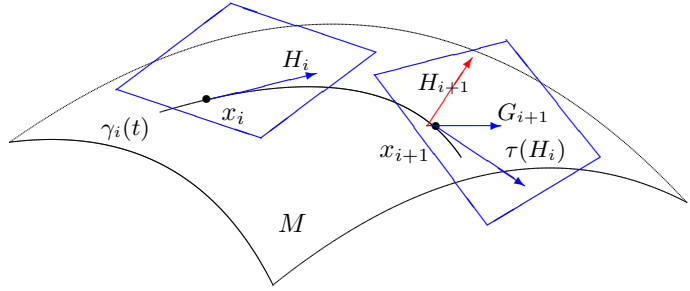
\begin{lemma}[cf. Fig. \ref{fig2}]\label{lemma:4}
With
\begin{equation}\label{eq:27}
   \mathcal{Z}(P):=\left[
  \begin{smallmatrix}
    0&-Z^\hop(P)\\Z(P)&0
  \end{smallmatrix}
\right],
\end{equation}
the geodesic distance from $\einheit$ to $P$ is given by
$
\dist(\einheit,P)=\| [\mathcal{Z}(P),\einheit] \|.
$
\end{lemma}

\begin{proof}
Let $\gamma(t)=\e^{ t \mathcal{Z}(P)}\einheit \e^{-t \mathcal{Z}(P)}$ be the 
geodesic \eqref{eq:geodesic2} emanating from $\einheit$ in direction
$\ad_{\mathcal{Z}(P)}(\einheit)$ 
with $\gamma(1)=P$. With $\dot{\gamma}(t)=\e^{ t \mathcal{Z}(P)}[\mathcal{Z}(P),\einheit]
\e^{-t \mathcal{Z}(P)}$ and by Eq. \eqref{eq:lengthfunctional} we get 
\begin{equation}
\textstyle\int_{0}^{1}\| \dot{\gamma}(t)\|\operatorname{d}t =
\textstyle\int_{0}^{1}\| [\mathcal{Z}(P),\einheit]\|\operatorname{d}t = \| [\mathcal{Z}(P),\einheit] \|=\sqrt{\tr [\mathcal{Z}(P),\einheit]^2} \label{eq:19}
\end{equation}
and the result follows.
\end{proof}

%%
%\begin{figure}[t]
% \begin{center}
%	\setlength{\unitlength}{1.5mm}
%	\begin{picture}(45, 33)
%	\qbezier(-25,17)(-1,20)(10,-2)
%	\qbezier(10,-2)(40,22)(65,9)
%	\qbezier(-25,17)(25,52)(65,9)
%	%
%	\qbezier(-5,21)(25,30)(40,15)
%	%
%	\textcolor{blue}{         
%		\put(5,23.5){\vector(4,1){27}}
%	}
%	\put(9,9){$\Gr_{m,n}$}
%	\put(-15,3){$\Herm_n$}
%	\put(5,23.5){\circle*{1}}
%	\put(30.8,21.3){\circle*{1}}
%	\put(9,29){\small{$\ad_{\mathcal{Z}(P)}(\einheit)$}}
%	\put(-5,19){$\gamma(0)=\einheit$}
%	\put(30,23){$\gamma(1)= P = \e^{\mathcal{Z}(P)}\einheit\e^{-\mathcal{Z}(P)}$}
%	\put(40,16){$\gamma(t)=\e^{ t \mathcal{Z}(P)}\einheit \e^{-t \mathcal{Z}(P)}$}
%     
%	\textcolor{white}{\put(0,22.7){\vector(1,2){4}}}
%%
%        \end{picture}         
%\end{center}
%\caption{Illustration of the result in Lemma \ref{lemma:4} which
%  states that 
%$\dist(\einheit,P)=\| [\mathcal{Z}(P),\einheit] \|$.}\label{fig2}
%\end{figure}
%%
Let $U \subset \Gr_{m,n}$ be a neighborhood around $\einheit$ and $B \subset T_{\einheit}\Gr_{m,n}$  such that
$\exp_{\einheit}\colon B \to U$ is one-one and onto. 
Consider the function
\begin{equation}
\label{eq:31xx}
f\colon U \to \R, \quad P \mapsto  \dist^2(\einheit,P).
\end{equation}
To calculate the derivative we use an equivalent expression of $f$, namely
\begin{equation}\tag{\ref{eq:31xx}$'$}
    f(P) = -\tr \left( \einheit \ad^2_{\mathcal{Z}(P)} ( \einheit ) \right).
\end{equation}
Let $H\in T_P\Gr_{m,n}$ be an arbitrary tangent vector.
For the
directional derivative we will use the abbreviation
$\mathcal{Z}':=\DD \mathcal{Z}(P)H$,
cf. Eq. \eqref{eq:directional_der}. Therefore
\begin{equation}
  \label{eq:4}
    \DD f(P)H=-2\tr\left(\einheit \ad_{\mathcal{Z}}(\ad_{\mathcal{Z}'} (\einheit) )\right).
\end{equation}
For computing the Riemannian gradient, we need an
expression for $\ad_{\mathcal{Z}'}$. 
To that end, note that Eq. \eqref{eq:2a} is equivalent to
\begin{equation}\label{eq:2}%\tag{\ref{eq:2a}$'$}
P = \e^{\ad_{\mathcal{Z}(P)}}(\einheit)
\end{equation}
and thus differentiating (\ref{eq:2}) 
with respect to $P$ in direction $H$ yields
\begin{equation}
  \label{eq:6aaa}
   \DD (P)H= H
=\e^{\ad_{\mathcal{Z}}}
\left(\big(
\tfrac{\id-\e^{-\ad_{\ad_{\mathcal{Z}}}}}{\ad_{\ad_{\mathcal{Z}}}}(\ad_{\mathcal{Z'}})\big) 
(\einheit) \right).
\end{equation}
Here, the operators $\e^{\ad_{\mathcal{Z}}}$ and
$\tfrac{\id-\e^{-\ad_{\ad_{\mathcal{Z}}}}}{\ad_{\ad_{\mathcal{Z}}}}$
have to be understood via their series expansion and acting on the right as usual.
By parallel transporting $\left[
  \begin{smallmatrix}
    0&K^\hop\\K&0
  \end{smallmatrix}
\right]\in T_{\einheit}\Gr_{m,n}$ with $K\in\C^{(n-m)\times m}$ along the unique geodesic connecting $\einheit$ and
$P$  it is easily seen that $H\in
T_P\Gr_{m,n}$ has the representation 
\begin{equation}
  \label{eq:10aaa}
  H= 
  \e^{\ad_{\mathcal{Z}}}\left[
  \begin{smallmatrix}
    0&K^\hop\\K&0
  \end{smallmatrix}
\right].
\end{equation}
 Using
\begin{equation}
\mathcal{K}:=\left[
  \begin{smallmatrix}
    0&-K^\hop\\K&0
  \end{smallmatrix}
\right]
\end{equation}
and (\ref{eq:10aaa}) a lengthy but straightforward computation including the
decomposition of the function
$x\mapsto\tfrac{1-\e^{-x}}{x}=\tfrac{\sinh x}{x}+\tfrac{1-\cosh x}{x}$ into even
and odd parts shows that
(\ref{eq:6aaa}) is equivalent to
\begin{equation}
  \label{eq:13aaa}
   \left[
      \begin{smallmatrix}
        0&K^\hop\\K&0
      \end{smallmatrix}
    \right]=\ad_{\mathcal{K}}(\einheit)=\left(\tfrac{\sinh\ad_{\ad_{\mathcal{Z}}}}{\ad_{\ad_{\mathcal{Z}}}}(\ad_{\mathcal{Z}'})\right) (\einheit).
\end{equation}
By assumption on $P$ being close enough to $\einheit$, the selfadjoint operator
    $\tfrac{\sinh\ad_{\ad_{\mathcal{Z}}}}{\ad_{\ad_{\mathcal{Z}}}}$ is
    invertible.  
Exploiting now the
representation property of the $\ad$-operator,
i.e. $[\ad_X,\ad_Y]=\ad_{[X,Y]}$, as well as linearity and
anti-selfadjointness, i.e. $\tr(A\ad_B(C))=-\tr(C\ad_B(A))$,  we can conclude that (\ref{eq:13aaa}) is equivalent to
    \begin{equation}
      \label{eq:14}
      \mathcal{Z}'=(\tfrac{\sinh\ad_{\mathcal{Z}}}{\ad_{\mathcal{Z}}})^{-1} (\mathcal{K}).
    \end{equation}
 In summary, 
\begin{equation}
  \label{eq:15}
  \begin{split}
   \DD f(P)H
&=-2\tr\!\Big( \einheit \ad_{\mathcal{Z}}\Big(\ad_{\left(\tfrac{\sinh\ad_{\mathcal{Z}}}{\ad_{\mathcal{Z}}}\right)^{-1} \!\!\!(\mathcal{K})}(\einheit) \Big) \Big)\\
&=2\tr\Big(\Big(\ad_{\einheit}\mathcal{Z}\Big) \Big(\ad_{\einheit}\Big(\Big(\tfrac{\sinh\ad_{\mathcal{Z}}}{\ad_{\mathcal{Z}}}\Big)^{-1}
    \!\!\!(\mathcal{K})\Big)\Big) \Big)\\
&=-2\tr\Big(\Big(\underbrace{\ad^2_{\einheit}\mathcal{Z}}_{=\mathcal{Z}}\Big)\Big(\Big(\tfrac{\sinh\ad_{\mathcal{Z}}}{\ad_{\mathcal{Z}}}\Big)^{-1}
    \!\!\!(\mathcal{K})\Big)\Big)\\
&=-2\tr\left(\mathcal{Z}\Big(\tfrac{\sinh\ad_{\mathcal{Z}}}{\ad_{\mathcal{Z}}}\Big)^{-1}
    \!\!\!(\mathcal{K})\right) = -2\tr(\mathcal{Z}\mathcal{K}).
 \end{split}
\end{equation}
The last equality in~\eqref{eq:15} is true by the self-adjointness of the operator
$(\tfrac{\sinh\ad_{\mathcal{Z}}}{\ad_{\mathcal{Z}}})^{-1}$
and $x\mapsto\tfrac{x}{\sinh x}=1-\tfrac{x^2}{6}+O(x^4)$ being an even function. In other
words, $(\tfrac{\sinh\ad_{\mathcal{Z}}}{\ad_{\mathcal{Z}}})^{-1}$ can
be considered to act 
as the identity operator to the left onto $\mathcal{Z}$ under the trace. Hence,
critical points are characterized by
\begin{equation}
  \label{eq:16}
    \DD f(P)=0\;\Longleftrightarrow \;
\Re\tr(Z^\hop K)=0\;\forall\; K\in\C^{(n-m)\times m}\;\Longleftrightarrow \;
Z=0.
\end{equation}
Together with Eq. \eqref{eq:2} this yields that
the unique critical point of $f$ is given by $P = \einheit$, as one would expect.\footnote{$f$ is defined in a neighborhood of $\einheit$ ensuring bijectivity of the Riemannian exponential.}
Moreover, from \eqref{eq:15} we can compute the Riemannian gradient of
the function $f$ by
\begin{equation}
  \label{eq:1}
  \begin{split}
    \DD f(P)H
    &=-2\tr(\mathcal{Z}\mathcal{K})=2\tr([\mathcal{Z},\einheit][\mathcal{K},\einheit])\\
    &=
    2\tr(\e^{\mathcal{Z}}[\mathcal{Z},\einheit]\e^{-\mathcal{Z}}\e^{\mathcal{Z}}[\mathcal{K},\einheit]\e^{-\mathcal{Z}})=\tr(2[\mathcal{Z},P]H).
  \end{split}
\end{equation}
Since $2[\mathcal{Z},P]\in T_P\Gr_{m.n}$ and since the trace is the Riemannian metric we can conclude, cf. \eqref{eq:directional_der}, that
\begin{equation}
  \label{eq:6}
    \operatorname{grad} f(P)= 2 [\mathcal{Z},P]. 
\end{equation}
Up to here all our computations were done sufficiently close to the
standard projector $\einheit$. As the Grassmannian is a homogeneous
    space, meaning that every
      point $P \in \Gr_{m,n}$ can be transformed to any other point on
      $Q \in \Gr_{m,n}$ by a suitable unitary matrix transformation $P= \Theta Q \Theta^\hop$, $\Theta \in U_n$, we can now
    transfer all our computations to an arbitrary element of
    $\Gr_{m,n}$. Let $Q\in\Gr_{m,n}$ be arbitrary, and let $P$ be sufficiently
close to
$Q$. Analogous to \eqref{eq:2}, we can then express 
$P=\e^{[\xi(P),Q]}Q\e^{-[\xi(P),Q]}$ for
unique $\xi\in T_Q\Gr_{m,n}$. Here, the tangent vector $\xi$ plays the role which $\left[
  \begin{smallmatrix}
    0&Z^\hop\\Z&0
  \end{smallmatrix}
\right]$ played in \eqref{eq:6}. By a slight abuse of notation we consider
the distance function between these \emph{arbitrary} $P$ and $Q$
\begin{equation}
  \label{eq:17}
   f\colon\Gr_{m,n}\to\R,\quad
P\mapsto \tr\xi\xi^\hop=\tr\xi^2=\dist^2(P,Q).
\end{equation}
Note, that by the above described invariance, the analogue of Eq. \eqref{eq:16} yields the critical point condition
\begin{equation}
  \label{eq:kritischerpunkt}
  \DD f(P)=0\Longleftrightarrow \xi=0.
\end{equation}

\begin{theorem}\label{thm:grad_dist}
The Riemannian gradient, with respect to
the Euclidean metric, of the function $f$, defined by (\ref{eq:17}), is given by
\begin{equation}
  \label{eq:22x}
  \operatorname{grad} f(P)= 2 \ad_{[\xi,Q]}P.
\end{equation}
\end{theorem}
Note that result \eqref{eq:22x} is in accordance with Proposition III 4.8
  in \cite{Sakai:1996lp}.
\begin{remark}
One can derive further explicit formulas for the distance, however,
they are less well suited
for gradient computations or numerics. Let $P,Q\in\Gr_{m,n}$. For any given $\Theta\in\U_{n}$ such that 
$  P=\Theta^{\hop}\einheit\Theta$ we define
\begin{equation}
    \left[\begin{smallmatrix}Q_{1} & Q_{2} \\ Q_{2}^{\hop} & Q_{3}\end{smallmatrix}\right]:=
  \Theta Q\Theta^{\hop}.\label{eq:29}
\end{equation}
Let $1\geq\lambda_{1}\geq\cdots\geq\lambda_{m}\geq 0$ denote the
 eigenvalues of $Q_{1}\in\Herm_m$. Then
\begin{equation}
  \label{eq:76}
  \dist(P,Q)=\textstyle{\sqrt{2\sum_{i=1}^{m}\arccos^{2}(\sqrt{\lambda_{i}})}}\ .
\end{equation}
Alternatively, let $1\!\geq\!\mu_{1}\!\geq\!\cdots\!\geq\!\mu_{n-m}\geq 0$ denote
the eigenvalues of $Q_{3}$. Then
\begin{equation}
  \label{eq:76aa}
  \dist(P,Q)=\textstyle{\sqrt{2\sum_{i=1}^{n-m}\arcsin^{2}(\sqrt{\mu_{i}})}}\ .
\end{equation}
In particular, if $P,Q\in\Gr_{m,n}$ with $Q=YY^{\hop}$ and $Y^{\hop}Y=I_m $, then
\begin{equation}
  \label{eq:76a}
  \half\dist^2(P,Q)=\tr
    \big(\arccos^{2}((Y^{\hop}PY)^{\frac{1}{2}})
    \big)\ .
\end{equation}
\end{remark}

\subsection{The Karcher mean}
We now consider a geodesically convex open ball\footnote{I.e. all points in $\mathcal{B}$ can be connected by a unique shortest geodesic contained completely
in $\mathcal{B}$. E.g. for a sphere, the maximal geodesically convex open balls are open hemispheres.}
$\mathcal{B}\subset\Gr_{m,n}$ containing all, say $N$ data points
$Q_i$. Note, that the Riemannian exponential map is bijective on $\mathcal{B}$ and 
thus the results from the last section carry over to the subsequent analysis. 
Moreover, this assumption ensures that the Karcher
  mean is the unique minimizer of the function defined by
  \eqref{eq:20}, \cite{karcher:77}. 
  This seems to be a sensible assumption in many applications,
where different data might be considered to be different measurements
of one and the same observable. Let us assume that $P\!\in\!\mathcal{B}$ and thus, for each $i$ there exists a unique $\xi_i\!\in\!
T_{Q_i}\!\Gr_{m,n}$ with
\begin{equation}
P=\exp_{Q_i}(\xi_i)=\e^{[\xi_i,Q_i]}Q_i\e^{-[\xi_i,Q_i]}.\label{eq:23}
\end{equation}
Let the \emph{Karcher mean} function now be defined as
\begin{equation}
  \label{eq:20}
    F\colon\mathcal{B}\to\R,\quad
P\mapsto\tfrac{1}{N}\textstyle\sum_{i=1}^N\dist^2(P,Q_i).
\end{equation}
Adapting \eqref{eq:31xx} and \eqref{eq:27} accordingly, we get
\begin{equation}
\label{eq:20b}
    F(P)=\tfrac{1}{N}\textstyle\sum_{i=1}^N\tr\xi_i^2=-\tfrac{1}{N}\textstyle\sum_{i=1}^N\tr Q_i\ad^2_{[\xi_i,Q_i]}Q_i.
\end{equation}
As a generalization of \eqref{eq:kritischerpunkt} we get the
well known fact \cite{karcher:77,CorcueraKendall:1999aa} that
\begin{equation}
  \label{eq:21}
  \DD F(P)=0\Longleftrightarrow \textstyle\sum_{i=1}^N\e^{[\xi_i,Q_i]}\xi_i\e^{-[\xi_i,Q_i]}=0.
\end{equation}
The interpretation of this condition is as follows. Let $P$ be
the unique critical point of $F$ on $\mathcal{B}$. Attaching a suitable
coordinate chart around $P$ tells us that in this chart
$P$ is equal to the usual Euclidean geometric mean of the data
points $Q_i$, expressed in exactly this chart.\footnote{Such a chart is called Riemannian
  normal coordinate chart in the literature.}
The Riemannian gradient of the Karcher mean now follows immediately from Theorem \ref{thm:grad_dist}.
\begin{theorem}
\label{thm:gradient}
The Riemannian gradient of $F$, defined by (\ref{eq:20}), is as
\begin{equation}
  \label{eq:22}
  \operatorname{grad} F(P)= \tfrac{2}{N}\textstyle\sum_{i=1}^N\ad_{[\xi_i,Q_i]}P.
\end{equation}
\end{theorem}

\subsection{Inverse of the Riemannian exponential}
\label{sec:53}
In the sequel we will present a procedure to explicitly compute 
the inverse of the Riemannian exponential.

To that end, let $Q \in \Gr_{m,n}$ and $B \subset T_Q \Gr_{m,n}$ be a neighborhood of $0$ such that
$\exp_Q \colon B \to \exp_Q(B)$ is a bijection. Assume further that $P \in
\exp_Q(B)$. Thus, there exists a unique $\xi \in B$ such that
\begin{equation}
  \label{eq:24}
  P=\e^{[\xi,Q]}Q\e^{-[\xi,Q]}.
\end{equation}
From the previous section, it follows that
$\tr\xi^2=\dist^2(P,Q)$. 
A partial task in our optimization
procedure will therefore be to compute $\xi$ as a
function of $P$ for a given $Q$ in \eqref{eq:24}. Let $\Theta\in\U_n$ such that $Q=\Theta \einheit\Theta^\hop$. It follows that 
\begin{equation}
  \label{eq:25}
  \Theta^\hop P \Theta=\e^{[\widehat{\xi},\einheit ]} \cdot\einheit \cdot\e^{-[\widehat{\xi},\einheit ]}=:\widehat{P}=\left[
  \begin{smallmatrix}
    \widehat{P}_{11}&\widehat{P}_{12}\\\widehat{P}_{12}^\hop&\widehat{P}_{22}
  \end{smallmatrix}
\right],
\end{equation}
with $\widehat{\xi}$ of the form
\begin{equation}
  \label{eq:26}
  \widehat{\xi}=\Theta^\hop \xi \Theta=\left[
  \begin{smallmatrix}
    0&-Z^\hop\\Z&0
  \end{smallmatrix}
\right].
\end{equation}
Using a singular value decomposition (SVD) as $
  Z^\hop\!=\!U\Sigma^\top\!V^\hop\!$
we arrive at the representation
\begin{equation}
  \begin{split}
    \label{eq:28}
    \left[
      \begin{smallmatrix}
        U^\hop&0\\0&V^\hop
      \end{smallmatrix}
    \right]\!\widehat{P}\!\left[
      \begin{smallmatrix}
        U&0\\0&V
      \end{smallmatrix}
    \right]\!&=\!\e^{\left[
      \begin{smallmatrix}
        0&-\Sigma^\top\\\Sigma&0
      \end{smallmatrix}
    \right]} \cdot \einheit \cdot \e^{-\left[
      \begin{smallmatrix}
        0&-\Sigma^\top\\\Sigma&0
      \end{smallmatrix}
    \right]}\\
&=\!\!\left[
      \begin{smallmatrix}
        \cos\sqrt{\Sigma^\top\Sigma}&
-\Sigma^\top\tfrac{\sin\sqrt{\Sigma\Sigma^\top}}{\sqrt{\Sigma\Sigma^\top}}\\
\tfrac{\sin\sqrt{\Sigma\Sigma^\top}}{\sqrt{\Sigma\Sigma^\top}}\Sigma&\cos\sqrt{\Sigma\Sigma^\top}
      \end{smallmatrix}
    \right]\!\!\cdot \! \einheit\!\cdot \!\!
\left[
      \begin{smallmatrix}
        \cos\sqrt{\Sigma^\top\Sigma}&
\Sigma^\top\tfrac{\sin\sqrt{\Sigma\Sigma^\top}}{\sqrt{\Sigma\Sigma^\top}}\\
-\tfrac{\sin\sqrt{\Sigma\Sigma^\top}}{\sqrt{\Sigma\Sigma^\top}}\Sigma&\cos\sqrt{\Sigma\Sigma^\top}
      \end{smallmatrix}
    \right]\\
&=\!\!\left[
      \begin{smallmatrix}
        \cos^2\sqrt{\Sigma^\top\Sigma}&
-\Sigma^\top\tfrac{\cos\sqrt{\Sigma\Sigma^\top}\sin\sqrt{\Sigma\Sigma^\top}}{\sqrt{\Sigma\Sigma^\top}}\\
\tfrac{\cos\sqrt{\Sigma\Sigma^\top}\sin\sqrt{\Sigma\Sigma^\top}}{\sqrt{\Sigma\Sigma^\top}}\Sigma&\sin^2\sqrt{\Sigma\Sigma^\top}
      \end{smallmatrix}
    \right]\\
&=\!\!\left[
      \begin{smallmatrix}
        \cos^2\sqrt{\Sigma^\top\Sigma}&
-\Sigma^\top\tfrac{\sin 2\sqrt{\Sigma\Sigma^\top}}{2\sqrt{\Sigma\Sigma^\top}}\\
\tfrac{\sin 2\sqrt{\Sigma\Sigma^\top}}{2\sqrt{\Sigma\Sigma^\top}}\Sigma&\sin^2\sqrt{\Sigma\Sigma^\top}
      \end{smallmatrix}
    \right].
  \end{split}
\end{equation}
The above matrix valued trigonometric functions have to be
interpreted via their series expansion. To be more precise, the
matrix $\Sigma$ is a rectangular diagonal matrix, and therefore
$\Sigma\Sigma^\top$, $\Sigma^\top\Sigma$, and their corresponding square roots are
diagonal (but square) as well. For $i\leq m$ the $ii-$th entry of
$\cos^2\sqrt{\Sigma^\top\Sigma}$ equals the squared cosine of the
$i-$th singular value of $\Sigma$, i.e. it equals $\cos^2 \sigma_i$,  
and equals $1$ otherwise. The right-hand side of \eqref{eq:28} can be efficiently computed via a
CS-decomposition \cite{GolubVan-Loan:1996aa} of $Y$, where $\widehat{P}\!=\!YY^\hop$ with $Y^\hop Y\!=\!I_m$,
or equivalently by an SVD of $\widehat{P}_{12}$. Using inverse trigonometric functions, $\xi$
can now be constructed from \eqref{eq:26}. Ultimately, we have explicitly 
constructed the \emph{inverse} of the Riemannian exponential map on $\Gr_{m,n}$.

%%%%%%%%%%%%%%%%%%%%%%%%%%%%%%%%%%%%%%%%%%%%%%%%%%%%%%%%%%%%%%%%%%%
%%                           SECTION 6                           %%
%%%%%%%%%%%%%%%%%%%%%%%%%%%%%%%%%%%%%%%%%%%%%%%%%%%%%%%%%%%%%%%%%%%
\section{A conjugate gradient algorithm for computing the Karcher mean on the Grassmannian}
\label{sec:06}
In the last sections, all ingredients have been derived for a geometric CG 
algorithm as described in Section \ref{sec:03}.
Here, we focus on its implementation and provide an explicit pseudo code 
for computing the Karcher mean of a set of complex subspaces.
Note that a 
projector $P$ can be uniquely expressed as $P=X X^\hop$, where $X$ is
an element of the complex Stiefel manifold
\begin{equation}
\St_{m,n} := \{ X \in \mathbb{C}^{n \times m} | 
X^{\hop} X = I_{m} \}.
\end{equation}
In order to compute the gradient of the Karcher mean function, the results in Section~\ref{sec:05} require that after computing the tangent
directions $\xi_{i} \in T_{Q_{i}}\!\Gr_{m,n}$ in \eqref{eq:23}, 
one needs to parallel transport them back to $T_{P}\Gr_{m,n}$, i.e.
$\e^{[\xi_i,Q_i]}\xi_{i}\e^{-[\xi_i,Q_i]} \in T_{P}\Gr_{m,n}$. 
The gradient of the Karcher mean function is simply the sum of all the parallel
transported vectors in $T_{P}\Gr_{m,n}$ according to
\eqref{eq:22}.
\begin{algorithm}
\caption{Riemannian gradient of the Karcher mean on $\Gr_{m,n}$}\label{algo:logGr}
	\SetAlgoNoLine
	\SetKwHangingKw{IN}{Input~~:}
	\SetKwHangingKw{SI}{Step~1~:}
	\SetKwHangingKw{SII}{Step~2~:}
	\SetKwHangingKw{OUT}{Output:}
	
	\IN{$Y_{i} \in \St_{m,n}$ for $i \!=\! 1,\ldots,N$ and $[X_{1}~X_{2}] \in \U_{n}$
		with $X_{1} \in \St_{m,n}$}
	
	\SI{ {\bf for} $i =  1,\ldots,N$ {\bf do}}
	
		{\hspace{25mm}{Compute the SVD of $X_{1}^{\hop} Y_{i} {Y_{i}}^{\hop} 
			X_{2} = U_{i}\Sigma_{i} {V_{i}}^{\hop}$}\;}
			
		{\hspace{25mm}Compute $\Lambda_{i} := U_{i}^{\hop} X_{1}^{\hop} Y_{i} 
			{Y_{i}}^{\hop} X_{1} U_{i}$}
			
	\SII{Compute $\mathcal{Z} = \sum\limits_{i=1}^{N}\big[
		\begin{smallmatrix}
			0&-Z_{i}^\hop\\Z_{i}&0
		\end{smallmatrix} \big]$ with $-Z_{i}^{\hop} = U_{i} \left[\arccos
			\sqrt{\Lambda_{i}}~~\pmb{0}\right] {V_{i}}^{\hop}
			 \in \mathbb{C}^{m \times (n-m)}$}
	
	\OUT{The Riemannian gradient $\operatorname{grad}F(X_{1}X_{1}^{\hop}) = -\big[X_{1}~X_{2}\big] \mathcal{Z} \big[X_{1}~X_{2}\big]^{\hop}$
		}
\end{algorithm}

Now let $N$ complex subspaces $\{Q_{i}\} \subset \Gr_{m,n}$ be given, and 
let $Y_{i1} \in \St_{m,n}$, for $i = 1, \ldots, N$, be the respective set of 
$N$ unitary bases.
For a given initialization $P \in \Gr_{m,n}$ with its representation $X \in \St_{m,n}$, 
we summarize a CG algorithm for computing the Karcher mean of the $Q_{i}$'s in 
Algorithm~\ref{algo:karcher}.

\begin{algorithm}
\caption{A CG for computing the Karcher mean on $\Gr_{m,n}$}\label{algo:karcher}
	\SetAlgoNoLine
	\SetKwHangingKw{IN}{Input~~:}
	\SetKwHangingKw{SI}{Step~1~:}
	\SetKwHangingKw{SII}{Step~2~:}
	\SetKwHangingKw{SIII}{Step~3~:}
	\SetKwHangingKw{SIV}{Step~4~:}
	\SetKwHangingKw{SV}{Step~5~:}
	\SetKwHangingKw{SVI}{Step~6~:}
	\SetKwHangingKw{SVII}{Step~7~:}
	\SetKwHangingKw{OUT}{Output:}

	\IN{Stiefel matrices $Y_{i} \in \St_{m,n}$ for $i=1,\ldots,N$}
			
	\SI{Generate an initial guess $[X_{1}^{(1)}~X_{2}^{(1)}] \in \U_{n}$ and 
			set $i=1$}
	
	\SII{Compute $H^{(1)} = -\grad F(X^{(1)} X^{(1)\hop})$ using 
		Algorithm~\ref{algo:logGr} }
	
	\SIII{Set $i=i+1$}
	
	\SIV{Update $[X_{1}^{(i+1)}~X_{2}^{(i+1)}] \gets \mathsf{e}^{a 
			[H_i,X_{1}^{(i)} {X_{1}^{(i)}}^\hop]} [X_{1}^{(i)}~X_{2}^{(i)}]$, where $
			a$ is computed via backtracking line search as in Algorithm~\ref{algo:backtracking}}
			
	\SV{Update $H^{(i+1)} \gets -G^{(i+1)} + r~G_{H_{i}}(a)$, where
		\vspace{-3mm}
			$$G^{(i+1)}=\grad F(X^{(i+1)} {X^{(i+1)}}^\hop),\vspace{-3mm}$$
			and $r$ is chosen according to Eq. (\ref{eq:different gammas})}
	\SVI{If $i\mod (2m(n-m)-1)=0$, set $H^{(i+1)} \gets -G^{(i+1)}$}

	\SVII{If $\left\| G^{(i+1)} \right\|$ is small enough, stop. 
		Otherwise, go to Step 3}
\end{algorithm}

As mentioned in Section~\ref{sec:03}, instead of employing a backtracking line search
for selecting an optimal step size at each conjugate direction in Step~4 in Algorithm~3, one can use a one
dimensional Newton step instead.
Applying this approach to a 
general Grassmannian requires the calculation of the first and second derivatives of 
eigenvalues and eigenvectors of a Hermitian matrix valued function 
$Y_{i}^{\hop} P(t) Y_{i}$, cf. \eqref{eq:76}.
Unfortunately, this approach is not well defined when the corresponding 
eigenvalues are multiple.

In the rest of this section, we derive a Newton step size selection as in 
\eqref{def:approx_stepsize_newton} for the special case where $m=1$, i.e. the 
complex projective space $\mathbb{CP}^{n-1} = \Gr_{1,n}$.
Let $Q_i=y_i y_i^\hop \in \mathbb{CP}^{n-1}$, $i=1,\ldots,N$ be a given set of data points 
with $y_{i} \in \mathbb{C}^{n}$ statisfying $\|y_{i}\| = 1$.
We define 
\begin{equation}
%\label{eq:8}
	\lambda_i(t):=y_i^\hop P(t)y_i,
\end{equation}
where $P(t)$ is defined by \eqref{eq:geodesic2},
and by abuse of notation we set $\lambda_i=\lambda_i(0)$. The first
and second derivatives of the Karcher mean $F$ can be computed as 
\begin{equation}
	\tfrac{\operatorname{d}}{\operatorname{d}\varepsilon}(F\circ P)
	(\varepsilon)|_{\varepsilon=0}
	= \operatorname{tr}(\operatorname{grad}F(P(0)) \dot{P}(0)),
\end{equation}
and
\begin{equation}
\begin{split}
	 \tfrac{\operatorname{d}^2}{\operatorname{d}\varepsilon^2}(F\circ P)
	(\varepsilon)|_{\varepsilon=0}
	&= \tfrac{\operatorname{d}^2}
	{\operatorname{d}\varepsilon^2}2\textstyle\sum_{i=1}^n\arccos^2\left.
	\sqrt{\lambda_i(\varepsilon)}\right|_{\varepsilon=0} \\
	&=2\tfrac{\operatorname{d}}{\operatorname{d}\varepsilon}
	\textstyle\sum_{i=1}^n\arccos\sqrt{\lambda_i(\varepsilon)}
	\tfrac{-\dot\lambda_i(\varepsilon)}{\sqrt{\lambda_i(\varepsilon)-\lambda_i
	(\varepsilon)^2}}\big|_{\varepsilon=0}\\
	&=2\textstyle\sum_{i=1}^n\left(\tfrac{\dot\lambda_i^2}{\lambda_i-\lambda_i^2}-
	\tfrac{{2(\lambda_i-\lambda_i^2)}\ddot\lambda_i-
	{\dot\lambda_i^2(1-2\lambda_i)}}
	{2(\sqrt{\lambda_i-\lambda_i^2})^{3}}\arccos\sqrt{\lambda_i}\right),
\end{split}
\end{equation}
respectively.
Finally, a one dimensional Newton step can be computed according to 
Eq.~\eqref{def:approx_stepsize_newton}.

%%%%%%%%%%%%%%%%%%%%%%%%%%%%%%%%%%%%%%%%%%%%%%%%%%%%%%%%%%%%%%%%%%%
%%                           SECTION 7                           %%
%%%%%%%%%%%%%%%%%%%%%%%%%%%%%%%%%%%%%%%%%%%%%%%%%%%%%%%%%%%%%%%%%%%
\section{An application in blind identification}
\label{sec:07}

We outline how the above described computation of the Karcher mean 
can be applied to the problem of \emph{Blind Identification (BI)}.
A simple instantaneous BI model assumes that the 
observation is a linear combination of some unknown sources, i.e.
\begin{equation}
\label{eq:bss}
	w(t) = A s(t),
\end{equation}
where $A \in \mathbb{C}^{n \times n}$ is the full rank system matrix and $w(t)$ presents $n$ observed linear mixtures of $n$ sources $s(t)$.
Blind identification aims to estimate the system matrix $A$ and 
various algorithms have been developed for this task, cf. \cite{Cichocki2002,como:book10}.
Recall, that the 
system matrix $A$ can be estimated only up to an arbitrary complex scaling
and permutation of the columns, cf. \cite{como:book10}.
Thus it is reasonable to consider estimates of columns of $A$ as elements in the
complex projective space $\mathbb{CP}^{n-1}:=\Gr_{1,n}$. We refer to
\cite{shen:ica10} for an approach on how to include the full rank constraint of $A$ into this setting.

It is known that performance of BI methods are sensitive
to the distribution of noise, cf. \cite{card:ieef93,waxm:ieeespl97}.
In other words, different distributions for the noise might lead to different 
optimal estimation of the system matrix. 
In particular, when system noise is present that varies over time, the matrix 
$A$ can no longer be guaranteed to be estimated correctly by a single process.
To overcome this difficulty, an intuitive idea is to simply average over
sub-optimal estimations of the system.
A similar approach has been investigated in \cite{mant:dsp06}, where a Kar\-cher mean 
based method is proposed for solving the real and whitened ICA problem.
Unfortunately, its applications are limited to the cases with stationary 
signals and noise.

Our experiments employ the following noisy model
\begin{equation}
	w_{i}(t) = (A + \epsilon_{i}Z_{i}) s_{i}(t), \qquad i=1,\dots, N,
\end{equation}
where $A$ is the ground truth system matrix, $Z_{i} \in \mathbb{C}^{n \times n}$
models system noise, $\epsilon_{i} > 0$ represents the noise level, 
and $s_{i}(t)$ denotes the unknown signals.
Both, real and imaginary parts of all entries of the ground truth system matrix $A$
are drawn from a normal distribution.
Perturbations $Z_{i}$ are applied to the system matrix, where the real and imaginary part of each entry of $Z_i$ are drawn from a uniform distribution on the interval $[-0.5, 0.5]$.

A popular BI algorithm is the so-called Strong Uncorrelated Transform (SUT),
cf. \cite{erik:mlsp04}. 
It uses the assumption that the source signals are uncorrelated and non-circular with 
distinct circularity coefficients. A joint diagonalizer of both the covariance and the pseudo-covariance matrix of the observations serves as an estimation of the inverse
of the system matrix $A$.
We employ the SUT for each of the $N$ subproblems to 
get $N$ estimates of $A$. 
After applying a suitable preprocess, we may assume that corresponding columns of 
the estimates are aligned, cf. \cite{maki:book07} for more details, and
thus we neglect the permutation ambiguity in our experiments.
The overall estimate is  
given by column wise computing the Karcher mean of the solutions of the sub-problems.
Identification performance of the proposed method
is measured by the normalized Amari error, cf. 
\cite{amar:nips96}, defined as
\begin{equation}
% \label{eq:41}
 J(\widehat{A},A) := \tfrac{1}{m} \left( \textstyle\sum\limits_{i=1}^{m} \tfrac
 {\sum\limits_{j=1}^{m} |b_{ij}|}
 {\max\limits_{j} |b_{ij}|}+ \textstyle\sum\limits_{j=1}^{m} \tfrac
 {\sum\limits_{i=1}^{m} |b_{ij}|}
 {\max\limits_{i} |b_{ij}|}\right) - 2,
\end{equation}
where $\widehat{A}$ is an estimation of $A$, and $B =(b_{ij})_{i,j=1}^{m}= 
\widehat{A}^{-1} A$.
In general, the smaller the Amari error, the better the identification of the system matrix.

In our experiments, we compare the proposed Karcher mean method to
a simple standard approach, referred to as the Euclidean mean approach.
Thereby, all solutions produced by SUT are summed up and the columns of the
obtained matrix are normalized to have unit norm.
First of all, we investigate the performance of both methods against the number of
estimations $N$.
We fix $n=5$ and run $N = 100$ experiments per number of estimations.
As the box plots of Amari errors in Figure~\ref{fig:03} and Figure~\ref{fig:04} suggest,
both Karcher and Euclidean subspace averaging methods admit a consistently 
increasing performance with an increasing number of estimations. 
%is plotted against the noise level $\epsilon_{i}$.
%in Figure~\ref{fig:03}. 
In our second experiment, we choose a various number of noise levels 
$\epsilon_{i} \in \{ 1, 0.5, 0.2, 0.1, 0.01\}$ and fix the number of estimations
$N = 10$. 
As shown in Figure~\ref{fig:05}, the Karcher mean approach outperforms the Euclidean counterpart consistently and its advantage is considerably higher, the more noise is present.
 \begin{figure}[t!]
 \centering
 \includegraphics[width=0.8\columnwidth]{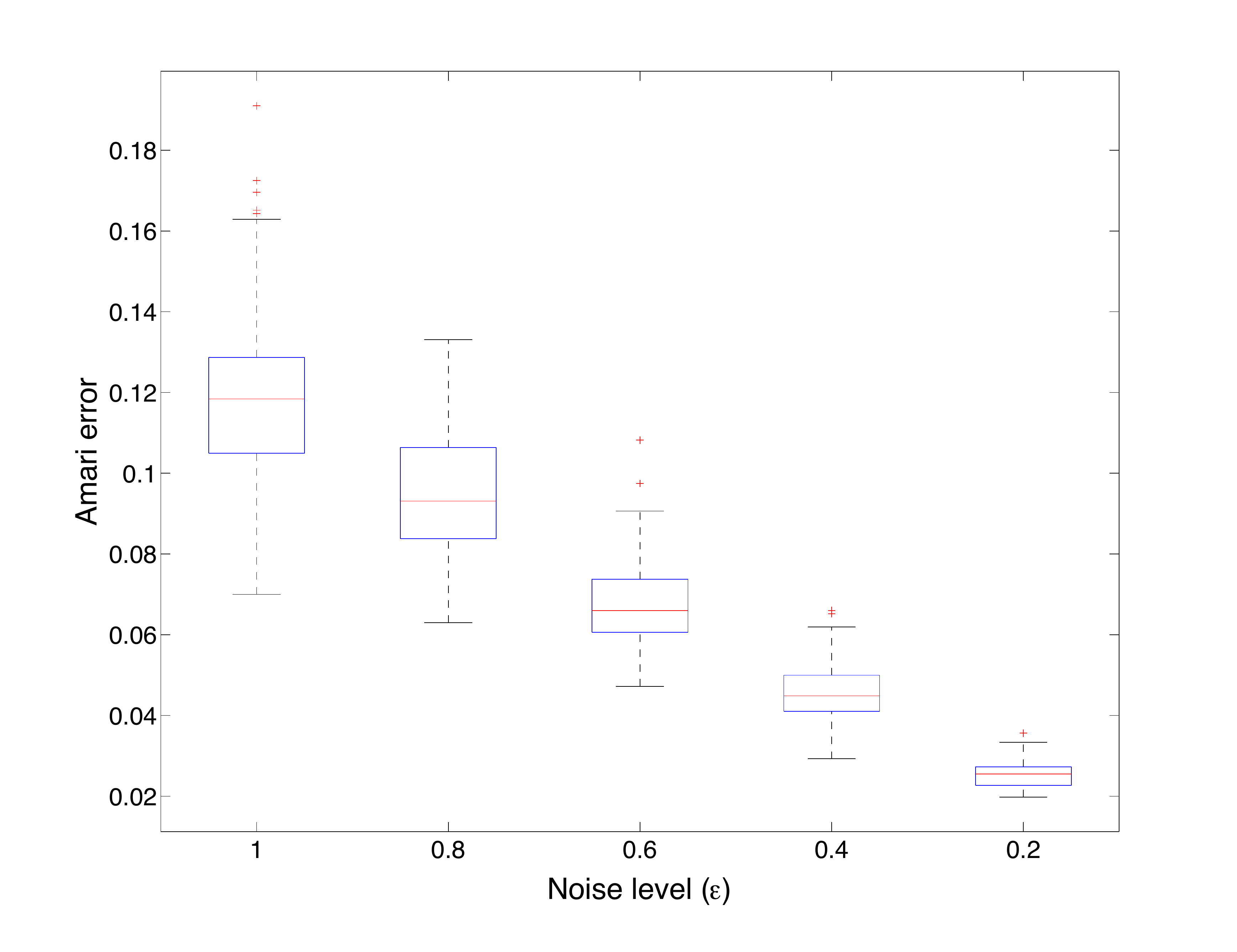}
 \vspace*{-4mm}
 \caption{Performance of subspace averaging via Karcher mean.}
 \label{fig:03}
 \end{figure} 
 \begin{figure}[t!]
 \centering
 \includegraphics[width=0.8\columnwidth]{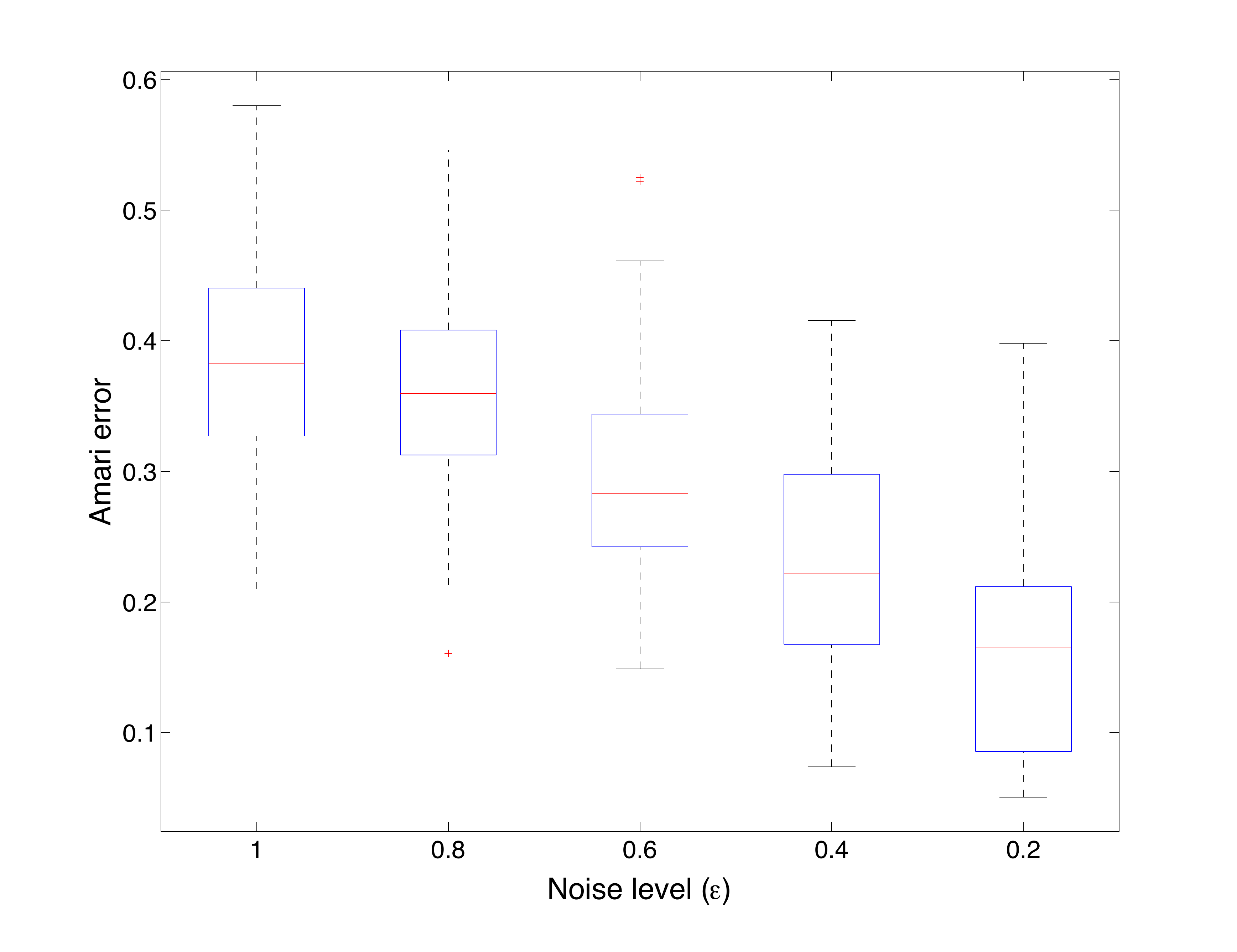}
 \vspace*{-4mm}
 \caption{Performance of subspace averaging via Euclidean mean.}
 \label{fig:04}
 \end{figure} 
 \begin{figure}[t!]
 \centering
 \includegraphics[width=0.8\columnwidth]{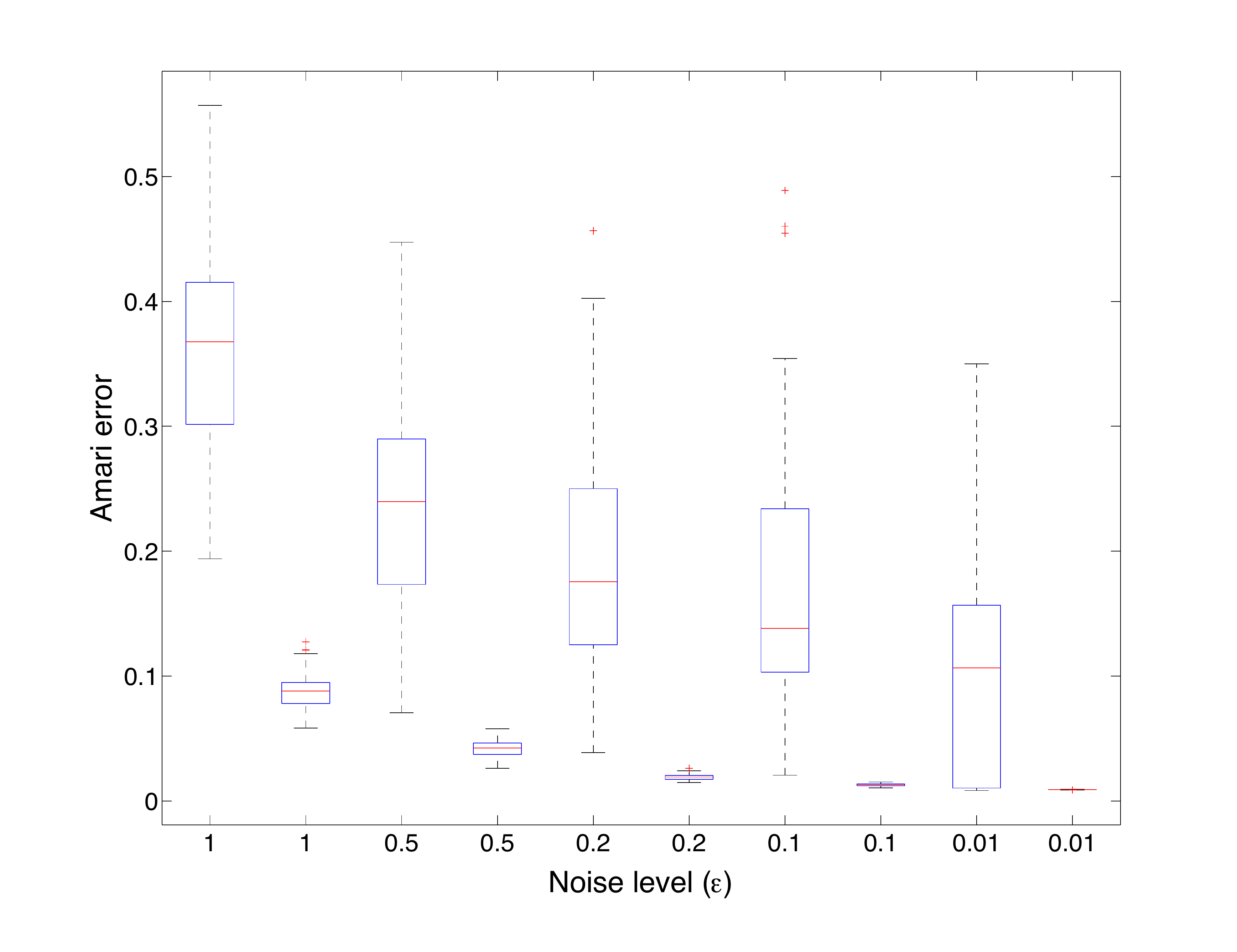}
 \vspace*{-4mm}
 \caption{Comparison of performance (left = Euclidean, right = Karcher).}
 \label{fig:05}
 \end{figure} 

%%%%%%%%%%%%%%%%%%%%%%%%%%%%%%%%%%%%%%%%%%%%%%%%%%%%%%%%%%%%%%%%%%%
%%                           SECTION 8                           %%
%%%%%%%%%%%%%%%%%%%%%%%%%%%%%%%%%%%%%%%%%%%%%%%%%%%%%%%%%%%%%%%%%%%
\section{Conclusion}
\label{sec:08}
This work focuses on the problem of averaging complex subspaces of
equal dimension by computing their Karcher mean, which is under mild assumptions the unique
minimum of a well defined
smooth function on the complex Grassmannian. 
An accessible introduction to the geometric structure of the Grassmanian is
provided by its identification with the set of Hermitian
projectors. In particular, explicit formulas for geodesics, parallel transport, and the Riemannian gradient 
of the Karcher mean function are given, which, in contrast to other formulas available in the literature, are well suited for implementation.

These results are used to propose an intrinsic conjugate gradient algorithm on the Grassmannian for computing the Karcher mean. We present experiments and outline the usability of such an approach in blind identification.

\end{document}